\documentclass[12pt,draft]{amsart}

\usepackage{a4,amssymb,cite}

\usepackage{graphicx,gastex}

\DeclareMathOperator{\Gr}{Gr}

\DeclareMathOperator{\lcm}{lcm}

\DeclareMathOperator{\Nil}{Nil}

\DeclareMathOperator{\var}{var}

\DeclareMathOperator{\ZR}{ZR}

\newtheorem{theorem}{Theorem}[section]

\newtheorem{proposition}[theorem]{Proposition}

\newtheorem{corollary}[theorem]{Corollary}

\newtheorem{lemma}[theorem]{Lemma}

\makeatletter

\renewcommand*\subjclass[2][2010]{\def\@subjclass{#2}\@ifundefined{subjclassname@#1}{\ClassWarning{\@classname}{Unknown edition (#1) of Mathematics Subject Classification; using '2010'.}}{\@xp\let\@xp\subjclassname\csname subjclassname@#1\endcsname}}

\renewcommand{\subjclassname}{\textup{2010} Mathematics Subject Classification}

\@addtoreset{equation}{section}

\makeatother

\begin{document}

\title[Upper-modular elements of the lattice of commutative varieties]{Upper-modular and related elements\\
of the lattice of commutative\\
semigroup varieties}

\thanks{Supported by the Ministry of Education and Science of the Russian Federation (project 2248), by grant of the President of the Russian Federation for supporting of leading scientific schools of the Russian Federation (project 5161.2014.1) and by Russian Foundation for Basic Research (grant 14-01-00524).}

\author{B.\,M.\,Vernikov}

\address{Ural Federal University, Institute of Mathematics and Computer Science, Lenina 51, 620000 Ekaterinburg, Russia}

\email{bvernikov@gmail.com}

\date{}

\begin{abstract}
We completely determine upper-modular, codistributive and costandard elements in the lattice of all commutative semigroup varieties. In particular, we prove that the properties of being upper-modular and codistributive elements in the mentioned lattice are equivalent. Moreover, in the nil-case the properties of being elements of all three types turn out to be equivalent.
\end{abstract}



\maketitle

\section{Introduction}
\label{section intr}

A remarkable attention in the theory of lattices is devoted to special elements of lattices. Recall definitions of several types of special elements. An element $x$ of the lattice $\langle L;\vee,\wedge\rangle$ is called
\begin{align*}
&\text{\emph{distributive} if}\quad&&\forall y,z\in L\colon\quad x\vee(y\wedge z)=(x\vee y)\wedge(x\vee z);\\
&\text{\emph{standard} if}\quad&&\forall y,z\in L\colon\quad(x\vee y)\wedge z=(x\wedge z)\vee(y\wedge z);\\
&\text{\emph{modular} if}\quad&&\forall y,z\in L\colon\quad y\le z\longrightarrow(x\vee y)\wedge z=(x\wedge z)\vee y;\\
&\text{\emph{upper-modular} if}\quad&&\forall y,z\in L\colon\quad y\le x\longrightarrow x\wedge(y\vee z)=y\vee(x\wedge z);
\end{align*}
\emph{neutral} if, for all $y,z\in L$, the elements $x$, $y$ and $z$ generate a distributive sublattice of $L$. \emph{Codistributive}, \emph{costandard} and \emph{lower-modular} elements are defined dually to distributive, standard and upper-modular elements respectively.

Special elements play an important role in the general lattice theory (see~\cite[Section~III.2]{Gratzer-11}, for instance). In particular, one can say that neutral elements are related with decompositions of a lattice into subdirect product of its intervals, while [co]distributive elements are connected with homomorphisms of a lattice into its principal filters [principal ideals]. Thus the knowledge of which elements of a lattice are neutral or [co]distributive gives essential information on the structure of the lattice as a whole.

There is a number of interrelations between the mentioned types of elements. It is evident that a neutral element is both standard and costandard; a [co]standard element is modular; a [co]distributive element is lower-modular [upper-modular]. It is well known also that a [co]standard element is [co]distributive (see~\cite[Theorem~253]{Gratzer-11}, for instance).

During last years, a number of articles appeared concerning special elements of the above mentioned types in the lattice \textbf{SEM} of all semigroup varieties and in certain its important sublattices, first of all, in the lattice \textbf{Com} of all commutative semigroup varieties. Briefly speaking, these articles contain complete descriptions of special elements of many types and essential information about elements of other types (including strong necessary conditions and descriptions in wide and important partial cases). These results are discussed in details in the recent survey~\cite{Vernikov-15+}. Special elements of the lattice \textbf{Com} are examined in the articles~\cite{Shaprynskii-11,Shaprynskii-12}. Results of these works give a complete description of neutral, standard, distributive or lower-modular elements of \textbf{Com} and a considerable information about its modular elements that reduces the problem of description of such elements to the nil-case. However, practically anything was unknown so far about costandard, codistributive or upper-modular elements of the lattice \textbf{Com}. A unique exclusion is a description of elements of these three types in the narrow and particular class of 0-reduced varieties that follows from~\cite[Proposition~2.3 and Theorem~1.2]{Shaprynskii-11}. In particular, it was unknown, whether the lattice \textbf{Com} contains costandard but not neutral elements, as well as upper-modular but not codistributive elements. Corresponding questions were formulated in~\cite{Vernikov-15+} (see Questions~4.11 and~4.12 there). For the sake of completeness, we mention that there exist codistributive but not costandard elements in the lattice \textbf{Com}. This fact can be easily deduced from results of~\cite{Shaprynskii-11} (see~\cite[Section~4.5]{Vernikov-15+}).

In this article, we completely determine costandard, codistributive or upper-modular elements in the lattice \textbf{Com}. In particular, we answer on Questions~4.11 and~4.12 of~\cite{Vernikov-15+}. Namely, we prove that, in this lattice, the properties of being upper-modular and codistributive elements are equivalent, but the properties of being costandard and neutral elements are not equivalent. Moreover, it turns out that all three properties we consider are equivalent for commutative nil-varieties. Note that these results extremely contrast with the situation in the lattice \textbf{SEM} where the properties of being upper-modular and codistributive elements are not equivalent (compare~\cite[Theorem~1.2]{Vernikov-08a} and~\cite[Theorem~1.2]{Vernikov-11}) but the properties of being costandard and neutral elements are equivalent~\cite[Theorem~1.3]{Vernikov-11}.

To formulate main results, we need some notation. We denote by $\mathcal T$, $\mathcal{SL}$ and $\mathcal{COM}$ the trivial variety, the variety of semilattices and the variety of all commutative semigroups respectively. If $\Sigma$ is a system of semigroup identities then $\var\Sigma$ stands for the semigroup variety given by $\Sigma$. For a natural number $m$, we put
$$\mathcal C_m=\var\{x^m=x^{m+1},\,xy=yx\}\ldotp$$
In particular, $\mathcal C_1=\mathcal{SL}$. For brevity, we put also $\mathcal C_0=\mathcal T$. Note that a semigroup $S$ satisfies the identity system $wx=xw=w$ where the letter $x$ does not occur in the word $w$ if and only if $S$ contains a zero element~0 and all values of $w$ in $S$ equal~0. We adopt the usual convention of writing $w=0$ as a short form of such a system. The main results of the article are the following two theorems.

\begin{theorem}
\label{umod&codistr}
For a commutative semigroup variety $\mathcal V$, the following are equivalent:
\begin{itemize}
\item[\textup{a)}]$\mathcal V$ is an upper-modular element of the lattice $\mathbf{Com}$;
\item[\textup{b)}]$\mathcal V$ is a codistributive element of the lattice $\mathbf{Com}$;
\item[\textup{c)}]one of the following holds:
\begin{itemize}
\item[\textup{(i)}]$\mathcal{V=COM}$,
\item[\textup{(ii)}]$\mathcal{V=M\vee N}$ where $\mathcal M$ is one of the varieties $\mathcal T$ or $\mathcal{SL}$, and $\mathcal N$ is a commutative variety satisfying the identities
\begin{align}
&\label{xxyz=0}x^2yz=0,\\
&\label{xxy=xyy}x^2y=xy^2,
\end{align}
\item[\textup{(iii)}]$\mathcal{V=G\vee M\vee N}$ where $\mathcal G$ is an Abelian periodic group variety, $\mathcal M$ is one of the varieties $\mathcal T$, $\mathcal{SL}$ or $\mathcal C_2$, and $\mathcal N$ is a commutative variety satisfying the identity
\begin{equation}
\label{xxy=0}
x^2y=0\ldotp
\end{equation}
\end{itemize}
\end{itemize}
\end{theorem}

\begin{theorem}
\label{costand}
For a commutative semigroup variety $\mathcal V$, the following are equivalent:
\begin{itemize}
\item[\textup{a)}]$\mathcal V$ is a modular and upper-modular element of the lattice $\mathbf{Com}$;
\item[\textup{b)}]$\mathcal V$ is a costandard element of the lattice $\mathbf{Com}$;
\item[\textup{c)}]one of the claims~\textup{(i)} or~\textup{(ii)} of Theorem~\textup{\ref{umod&codistr}} holds.
\end{itemize}
\end{theorem}

The article consists of four sections. In Section~\ref{aux} we collect auxiliary results used in what follows, Section~\ref{proof} is devoted to the proof of Theorems~\ref{umod&codistr} and~\ref{costand}, and Section~\ref{corollary} contains several corollaries from the main results.

\section{Preliminaries}
\label{aux}

We start with certain results about special elements in the lattice \textbf{Com} obtained earlier.

\begin{proposition}[\!\!{\cite[Theorem~1.2]{Shaprynskii-11}}]
\label{neutr}
A commutative semigroup variety $\mathcal V$ is a neutral element of the lattice $\mathbf{Com}$ if and only if $\mathcal{V=M\vee N}$ where $\mathcal M$ is one of the varieties $\mathcal T$ or $\mathcal{SL}$, and $\mathcal N$ is a commutative variety satisfying the identity~\eqref{xxy=0}.\qed
\end{proposition}

\begin{proposition}[\!\!{\cite[Theorem~1.4]{Shaprynskii-12}}]
\label{mod}
If a commutative semigroup variety $\mathcal V$ is a modular element of the lattice $\mathbf{Com}$ then $\mathcal{V=M\vee N}$ where $\mathcal M$ is one of the varieties $\mathcal T$ or $\mathcal{SL}$, and $\mathcal N$ is a nil-variety.\qed
\end{proposition}

It is generally known that the variety $\mathcal{SL}$ is an atom of the lattice \textbf{SEM} (and therefore, of the lattice \textbf{Com}). Proposition~\ref{neutr} implies that this variety is neutral in \textbf{Com}. Combining these facts with~\cite[Corollary~2.1]{Shaprynskii-11}, we have the following

\begin{lemma}
\label{join with SL}
A commutative semigroup variety $\mathcal V$ is an upper-modular \textup[costandard\textup] element of the lattice $\mathbf{Com}$ if and only if the variety $\mathcal{V\vee SL}$ has the same property.\qed
\end{lemma}

The subvariety lattice of a variety $\mathcal X$ is denoted by $L(\mathcal X)$. The following lemma is generally known.

\begin{lemma}
\label{dir prod}
If $\mathcal V$ is a semigroup variety with $\mathcal{V\nsupseteq SL}$ then the lattice $L(\mathcal{V\vee SL})$ is isomorphic to the direct product of the lattices $L(\mathcal V)$ and $L(\mathcal{SL})$.\qed
\end{lemma}

We denote by $F$ the free semigroup. The symbol $\equiv$ stands for the equality relation on $F$. If $u\in F$ then we denote by $c(u)$ the set of letters occurring in $u$ and by $\ell(u)$ the length of $u$. If $u,v\in F$ then we write $u\vartriangleleft v$ whenever $v\equiv a\xi(u)b$ for some (maybe empty) words $a$ and $b$ and some homomorphism $\xi$ of $F$. The first claim in the following lemma is evident, while the second one follows from~\cite[Lemma~1.3(iii)]{Vernikov-Volkov-00}.

\begin{lemma}
\label{splitting}
Let $\mathcal N$ be a nil-variety of semigroups.
\begin{itemize}
\item[\textup{(i)}]If $\mathcal N$ satisfies an identity $u=v$ with $c(u)\ne c(v)$ then $\mathcal N$ also satisfies the identity $u=0$.
\item[\textup{(ii)}]If $\mathcal N$ is commutative and satisfies an identity $u=v$ with $u\vartriangleleft v$ then $\mathcal N$ also satisfies the identity $u=0$.\qed
\end{itemize}
\end{lemma}

We need the following two technical corollaries from Lemma~\ref{splitting}. Put $W=\{x^2y,xyx,yx^2,y^2x,yxy,xy^2\}$.

\begin{corollary}
\label{xxy=w}
If a commutative nil-variety of semigroups  $\mathcal N$ satisfies an identity of the form $u=v$ where $u\in\{x^2y,xy^2\}$ and $v\notin W$ then $\mathcal N$ also satisfies the identity~\eqref{xxy=0}.
\end{corollary}

\begin{proof}
Suppose at first that $u\equiv x^2y$. If $c(v)\ne\{x,y\}$ then $\mathcal N$ satisfies the identity~\eqref{xxy=0} by Lemma~\ref{splitting}(i). Let now $c(v)=\{x,y\}$. If $\ell(v)<3$ then $v\vartriangleleft x^2y$ and Lemma~\ref{splitting}(ii) implies that $\mathcal N$ satisfies the identity~\eqref{xxy=0} again. If $\ell(v)=3$ then $v\in W$ contradicting the hypothesis. Finally, if $\ell(v)>3$ then it is easy to see that $v$ equals in $\mathcal{COM}$ (and therefore, in $\mathcal N$) to a word $v'$ such that $x^2y\vartriangleleft v'$. Now Lemma~\ref{splitting}(ii) applies again, and we conclude that $\mathcal N$ satisfies the identity~\eqref{xxy=0} as well.

The case when $u\equiv xy^2$ may be considered quite analogously with the conclusion that $\mathcal N$ satisfies the identity $xy^2=0$ that is equivalent to~\eqref{xxy=0} modulo the commutative law.
\end{proof}

\begin{corollary}
\label{xxy=xyy implies xxyz=0}
If a commutative nil-variety of semigroups  $\mathcal N$ satisfies the identity~\eqref{xxy=xyy} then $\mathcal N$ also satisfies the identity~\eqref{xxyz=0}.
\end{corollary}

\begin{proof}
Substituting $yz$ to $y$ in the identity~\eqref{xxy=xyy}, we obtain $x^2yz=x(yz)^2=xy^2z^2$. Since $x^2yz\vartriangleleft xy^2z^2$, it remains to refer to Lemma~\ref{splitting}(ii).
\end{proof}

A semigroup variety $\mathcal V$ is called a \emph{variety of degree} $n$ if all nilsemigroups in $\mathcal V$ are nilpotent of degree $\le n$ and $n$ is the least number with such property. A variety is said to be a \emph{variety of finite degree} if it has a degree $n$ for some $n$; otherwise, a variety is called a \emph{variety of infinite degree}. The following lemma follows from~\cite[Proposition~2.11]{Vernikov-08a} and~\cite[Theorem~2]{Tishchenko-90}.

\begin{lemma}
\label{commut fin deg}
A commutative semigroup variety $\mathcal V$ is a variety of degree $\le n$ if and only if it satisfies an identity of the form
\begin{equation}
\label{commut fin deg equ}
x_1x_2\cdots x_n=(x_1x_2\cdots x_n)^{t+1}
\end{equation}
for some natural number $t$.\qed
\end{lemma}

If $\mathcal V$ is a variety of finite degree then we denote the degree of $\mathcal V$ by $\deg(\mathcal V)$; otherwise, we write $\deg(\mathcal V)=\infty$.

\begin{corollary}
\label{deg of join}
If $\mathcal X$ and $\mathcal Y$ are commutative semigroup varieties then
$$\mathcal{\deg(X\vee Y)=\max\bigl\{\deg(X),\deg(Y)\bigr\}}\ldotp$$
\end{corollary}

\begin{proof}
If at least one of the varieties $\mathcal X$ or $\mathcal Y$ has infinite degree then
$$\mathcal{\deg(X\vee Y)=\infty=\max\bigl\{\deg(X),\deg(Y)\bigr\}}\ldotp$$
Let now $\deg(\mathcal X)=k$ and $\deg(\mathcal Y)=m$. Lemma~\ref{commut fin deg} implies that the varieties $\mathcal X$ and $\mathcal Y$ satisfy, respectively, the identities
$$x_1x_2\cdots x_k=(x_1x_2\cdots x_k)^{r+1}$$
and
$$x_1x_2\cdots x_m=(x_1x_2\cdots x_m)^{s+1}$$
for some $r$ and $s$. Suppose without loss of generality that $k\le m$. Substitute $x_k\cdots x_m$ to $x_k$ in the first of the two mentioned identities. We have that $\mathcal X$ satisfies the identity
$$x_1x_2\cdots x_m=(x_1x_2\cdots x_m)^{r+1}\ldotp$$
Then $\mathcal{X\vee Y}$ satisfies the identity
$$x_1x_2\cdots x_m=(x_1x_2\cdots x_m)^{rs+1}\ldotp$$
Now Lemma~\ref{splitting}(ii) applies with the conclusion that
$$\deg(\mathcal{X\vee Y})\le m=\max\bigl\{\mathcal{\deg(X),\deg(Y)}\bigr\}\ldotp$$
Since the unequality $\max\bigl\{\mathcal{\deg(X),\deg(Y)}\bigr\}\le\deg(\mathcal{X\vee Y})$ is evident, we are done.
\end{proof}

The following statement follows from~\cite[Proposition~1]{Volkov-89} and results of the article~\cite{Head-68}.

\begin{lemma}
\label{decomposition}
If $\mathcal V$ is a commutative semigroup variety and $\mathcal{V\ne COM}$ then $\mathcal{V=G\vee C}_m\vee\mathcal N$ for some Abelian periodic group variety $\mathcal G$, some $m\ge0$ and some nil-variety $\mathcal N$.\qed
\end{lemma}

It is well known that an arbitrary periodic semigroup variety $\mathcal V$ contains the greatest group subvariety that we denote by $\Gr(\mathcal V)$. A semigroup variety $\mathcal V$ is called \emph{combinatorial} if all groups in $\mathcal V$ are singletons.

\begin{lemma}
\label{Gr of join}
If $\mathcal G$ is a periodic group variety and $\mathcal F$ is a combinatorial semigroup variety then $\mathcal{\Gr(G\vee F)=G}$.
\end{lemma}

\begin{proof}
The inclusion $\mathcal{G\subseteq\Gr(G\vee F)}$ is evident. To verify the converse inclusion, we suppose that the variety $\mathcal G$ satisfies the identity $u=v$. Being combinatorial, the variety $\mathcal F$ satisfies the identity $x^n=x^{n+1}$ for some natural $n$. Therefore $\mathcal{G\vee F}$ satisfies the identity $u^{n+1}v^n=u^nv^{n+1}$. If we reduce it on $u^n$ from the left and on $v^n$ from the right, we obtain that the identity $u=v$ holds in $\Gr(\mathcal{G\vee F})$.
\end{proof}

A semigroup is called \emph{combinatorial} if all its subgroups are singletons. It is easy to verify that the variety $\mathcal C_m$ is generated by the $(m+1)$-element combinatorial cyclic monoid. We will use this fact below without special references. It may be easily checked that the join of all varieties of the form $\mathcal C_m$ coincides with the variety $\mathcal{COM}$. Therefore, for a periodic semigroup variety $\mathcal X$ there exists the largest number $m\in\mathbb N\cup\{0\}$ with the property $\mathcal C_m\subseteq\mathcal X$. We denote this number by $m(\mathcal X)$. 

\begin{lemma}
\label{m(G+C_m+N)}
If $\mathcal G$ is a periodic group variety, $m\ge0$ and $\mathcal N$ is a nil-variety of semigroups then $m(\mathcal G\vee\mathcal C_m\vee\mathcal N)=m$.
\end{lemma}

\begin{proof}
The varieties $\mathcal G$ and $\mathcal N$ satisfy, respectively, the identities $x=x^{r+1}$ and $x^n=0$ for some natural $r$ and $n$. Whence the identity $x^my^n=x^{r+m}y^n$ holds in the variety $\mathcal G\vee\mathcal C_m\vee\mathcal N$. Substituting~1 to $y$ in this identity, we see that every monoid in $\mathcal G\vee\mathcal C_m\vee\mathcal N$ satisfies the identity $x^m=x^{r+m}$. Clearly, any combinatorial semigroup with this identity satisfies the identity $x^m=x^{m+1}$. Since every cyclic semigroup is commutative, we have $m(\mathcal G\vee\mathcal C_m\vee\mathcal N)\le m$. The converse unequality is evident.
\end{proof}

\begin{corollary}
\label{m(X+Y)}
If $\mathcal X$ and $\mathcal Y$ are periodic commutative semigroup varieties then $m(\mathcal{X\vee Y})=\max\bigl\{m(\mathcal X),m(\mathcal Y)\bigr\}$.
\end{corollary}

\begin{proof}
Lemma~\ref{decomposition} implies that $\mathcal{X=G}_1\vee\mathcal C_{m_1}\vee\mathcal N_1$ and $\mathcal{Y=G}_2\vee\mathcal C_{m_2}\vee\mathcal N_2$ for some Abelian periodic group varieties $\mathcal G_1$ and $\mathcal G_2$, some $m_1,m_2\ge0$ and some nil-varieties $\mathcal N_1$ and $\mathcal N_2$. Then $m(\mathcal X)=m_1$ and $m(\mathcal Y)=m_2$ by Lemma~\ref{m(G+C_m+N)}. We have
$$\mathcal{X\vee Y=G}_1\vee\mathcal G_2\vee\mathcal C_{m_1}\vee\mathcal C_{m_2}\vee\mathcal N_1\vee\mathcal N_2=\mathcal G_1\vee\mathcal G_2\vee\mathcal C_{\max\{m_1,m_2\}}\vee\mathcal N_1\vee\mathcal N_2\ldotp$$
Applying Lemma~\ref{m(G+C_m+N)} one more time, we have
$$m(\mathcal{X\vee Y})=\max\{m_1,m_2\}=\max\bigl\{m(\mathcal X),m(\mathcal Y)\bigr\}\ldotp$$
Corollary is proved.
\end{proof}

\section{Proofs of main results}
\label{proof}

To prove both theorems, it suffices to verify the implications a)\,$\longrightarrow$\,c) and c)\,$\longrightarrow$\,b) because the implications b)\,$\longrightarrow$\,a) in both theorems are evident.

\medskip

\emph{The implication} a)\,$\longrightarrow$\,c) \emph{of Theorem}~\ref{umod&codistr}. The article~\cite{Vernikov-08a} contains, among others, the proof of the following fact: if a periodic commutative semigroup variety $\mathcal V$ is an upper-modular element of the lattice \textbf{SEM} then one of the claims~(ii) or~(iii) of Theorem~\ref{umod&codistr} holds. Almost all varieties that appear in the corresponding fragment of~\cite{Vernikov-08a} are commutative. The unique exclusion is a periodic group variety $\mathcal G$ that appear in the verification of the following fact: every nil-subvariety of $\mathcal V$ satisfies the identity~\eqref{xxy=xyy}. There are no the requirement that the variety $\mathcal G$ is Abelian in~\cite{Vernikov-08a}. But if we add this requirement to arguments used in~\cite{Vernikov-08a} then the proof will remains valid. Thus, in actual fact, it is verified in~\cite{Vernikov-08a} that if $\mathcal V$ is an upper-modular element of the lattice \textbf{Com} and $\mathcal{V\ne COM}$ then $\mathcal V$ satisfies one of the claims~(ii) or~(iii) of Theorem~\ref{umod&codistr}.

\medskip

\emph{The implication} a)\,$\longrightarrow$\,c) \emph{of Theorem}~\ref{costand}. Let $\mathcal V$ be a modular and upper-modular element of the lattice \textbf{Com} and $\mathcal{V\ne COM}$. Then we may apply Proposition~\ref{mod} and conclude that $\mathcal{V=M\vee N}$ where $\mathcal M$ is one of the varieties $\mathcal T$ or $\mathcal{SL}$, and $\mathcal N$ is a nil-variety. The variety $\mathcal N$ is an upper-modular element in the lattice \textbf{Com} by Lemma~\ref{join with SL}. In view of the proved above implication a)\,$\longrightarrow$\,c) of Theorem~\ref{umod&codistr}, we have that $\mathcal N$ satisfies the identities~\eqref{xxyz=0} and~\eqref{xxy=xyy}. Thus the claim~(ii) of Theorem~\ref{umod&codistr} fullfills. 

\medskip

\emph{The implication} c)\,$\longrightarrow$\,b) \emph{of Theorem}~\ref{costand}. Let $\mathcal{V=M\vee N}$ where $\mathcal M$ is one of the varieties $\mathcal T$ or $\mathcal{SL}$, and $\mathcal N$ is a commutative variety satisfying the identities~\eqref{xxyz=0} and~\eqref{xxy=xyy}. We need to verify that $\mathcal V$ is costandard in \textbf{Com}. In view of Lemma~\ref{join with SL}, it suffices to check that the variety $\mathcal N$ is costandard in \textbf{Com}. Let $\mathcal X$ and $\mathcal Y$ be arbitrary commutative semigroup varieties. It suffices to verify that
$$\mathcal{(N\vee Y)\wedge(X\vee Y)\subseteq(N\wedge X)\vee Y}$$
because the converse inclusion is evident. If at least one of the varieties $\mathcal X$ or $\mathcal Y$ coincides with the variety $\mathcal{COM}$ then the desirable inclusion is evident. Thus we may assume that the varieties $\mathcal X$ and $\mathcal Y$ are periodic. Let $u=v$ be an arbitrary identity that is satisfied by $\mathcal{(N\wedge X)\vee Y}$. We need to verify that it holds in $\mathcal{(N\vee Y)\wedge(X\vee Y)}$. By the hypothesis, the identity $u=v$ holds in $\mathcal Y$ and there exists a deduction of this identity from the identities of the varieties $\mathcal N$ and $\mathcal X$. Let the sequence of words
\begin{equation}
\label{sequence}
u_0\equiv u,u_1,\dots,u_k\equiv v
\end{equation}
be the shortest such deduction. If $k=1$ then the identity $u=v$ holds in one of the varieties $\mathcal N$ or $\mathcal X$. Then it is satisfied by one of the varieties $\mathcal{N\vee Y}$ or $\mathcal{X\vee Y}$, whence by the variety $\mathcal{(N\vee Y)\wedge(X\vee Y)}$. Thus we may assume that $k>1$. If the identity $u=v$ holds in $\mathcal N$ then it holds in $\mathcal{N\vee Y}$ and therefore, in $\mathcal{(N\vee Y)\wedge(X\vee Y)}$. Thus we may assume that $u=v$ fails in $\mathcal N$. In particular, at least one of the words $u$ or $v$, say $u$, does not equal~0 in $\mathcal N$. Since $\mathcal N$ satisfies the identity~\eqref{xxyz=0}, this means that $u$ coincides with one of the words $x_1x_2\cdots x_n$ for some $n$, $x^2$, $x^3$ or $x^2y$. Further considerations are divided into three cases.

\medskip

\emph{Case} 1: $u\equiv x_1x_2\cdots x_n$. If $v\equiv x_{1\pi}x_{2\pi}\cdots x_{n\pi}$ for some non-trivial permutation $\pi$ on the set $\{1,2,\dots,n\}$ then the identity $u=v$ holds in the variety $\mathcal{COM}$ and therefore, in the variety $\mathcal{(N\vee Y)\wedge(X\vee Y)}$. Otherwise, Lemma~\ref{splitting} applies with the conclusion that every nilsemigroup in $\mathcal Y$ satisfies the identity
\begin{equation}
\label{x_1...x_n=0}
x_1x_2\cdots x_n=0\ldotp
\end{equation}
This means that $\mathcal Y$ is a variety of degree $\le n$. Now we may apply Lemma~\ref{commut fin deg} and conclude that $\mathcal Y$ satisfies the identity
$$x_1x_2\cdots x_n=(x_1x_2\cdots x_n)^{r+1}$$
for some natural $r$ and therefore, the identity
$$x_1x_2\cdots x_n=(x_1x_2\cdots x_n)^{r\ell+1}$$
for any natural $\ell$. Thus the words $x_1x_2\cdots x_n$, $(x_1x_2\cdots x_n)^{r\ell+1}$ (for all $\ell$) and $v$ pairwise equal each to other in the variety $\mathcal Y$. 

Further, one of the varieties $\mathcal N$ or $\mathcal X$ satisfies the identity $x_1x_2\cdots x_n=u_1$. If $v\equiv x_{1\pi}x_{2\pi}\cdots x_{n\pi}$ for some non-trivial permutation $\pi$ on the set $\{1,2,\dots,n\}$ then the identity $u=u_1$ holds in both varieties $\mathcal N$ and $\mathcal X$. This contradicts the claim that~\eqref{sequence} is the shortest deduction of the identity $u=v$ from the identities of the varieties $\mathcal N$ and $\mathcal X$. Repeating arguments from the previous paragraph, we may conclude that there exists a natural $s$ such that the words $x_1x_2\cdots x_n$, $(x_1x_2\cdots x_n)^{s\ell+1}$ (for all $\ell$) and $v$ pairwise equal each to other in one of the varieties $\mathcal N$ or $\mathcal X$. Then the words $x_1x_2\cdots x_n$, $(x_1x_2\cdots x_n)^{rs+1}$ and $v$ pairwise equal each to other in one of the varieties $\mathcal{N\vee Y}$ or $\mathcal{X\vee Y}$ and therefore, in the variety $\mathcal{(N\vee Y)\wedge(X\vee Y)}$. In particular, the variety $\mathcal{(N\vee Y)\wedge(X\vee Y)}$ satisfies the identity $u=v$.

\medskip

\emph{Case} 2: $u\equiv x^2$ or $u\equiv x^3$. One can verify the desirable statement in slightly more general situation when $u\equiv x^n$ for some $n$. (In actual fact, this statement is evident whenever $n>3$ because the variety $\mathcal N$ satisfies the identity $x^4=0$. But our considerations below does not depend on $n$.) The identity $x^n=v$ holds in $\mathcal Y$. Then Lemma~\ref{splitting} implies that every nilsemigroup in $\mathcal Y$ satisfies the identity $x^n=0$. Being periodic, the variety $\mathcal Y$ satisfies the identity $x^p=x^q$ for some natural $p$ and $q$ with $p<q$. Let $p$ be the least number with such a property. In view of Lemma~\ref{splitting}, each nilsemigroup in $\mathcal Y$ satisfies the identity $x^p=0$. Clearly, $p$ is the least number with such a property. Therefore $n\ge p$. Multiplying the identity $x^p=x^q$ on $x^{n-p}$, we see that $\mathcal Y$ satisfies the identity $x^n=x^{n+r}$ for some $r$ and therefore, the identity $x^n=x^{n+r\ell}$ for every natural $\ell$. Thus the words $x^n$, $x^{n+r\ell}$ (for all $\ell$) and $v$ pairwise equal each to other in the variety $\mathcal Y$. 

Further, one of the varieties $\mathcal N$ or $\mathcal X$ satisfies the identity $x^n=u_1$. The same arguments as in the previous paragraph show that there exists a natural $s$ such that the words $x^n$, $x^{n+s\ell}$ (for all $\ell$) and $v$ pairwise equal each to other in one of the varieties $\mathcal N$ or $\mathcal X$. Then the words $x^n$, $x^{n+rs}$ and $v$ pairwise equal each to other in one of the varieties $\mathcal{N\vee Y}$ or $\mathcal{X\vee Y}$ and therefore, in the variety $\mathcal{(N\vee Y)\wedge(X\vee Y)}$. In particular, the variety $\mathcal{(N\vee Y)\wedge(X\vee Y)}$ satisfies the identity $u=v$.

\medskip

\emph{Case} 3: $u\equiv x^2y$. It is well known that every periodic semigroup variety $\mathcal W$ contains the greatest nil-subvariety. We denote this subvariety by $\Nil(\mathcal W)$. In view of Lemma~\ref{decomposition}, $\mathcal{X=G}_1\vee\mathcal C_{m_1}\vee\mathcal N_1$ and $\mathcal{Y=G}_2\vee\mathcal C_{m_2}\vee\mathcal N_2$ for some Abelian periodic group varieties $\mathcal G_1$ and $\mathcal G_2$, some $m_1,m_2\ge0$ and some nil-varieties $\mathcal N_1$ and $\mathcal N_2$. We may assume without loss of generality that $\mathcal N_1=\Nil(\mathcal X)$ and $\mathcal N_2=\Nil(\mathcal Y)$.  

If the variety $\mathcal N$ satisfies the identity~\eqref{xxy=0} then
$$\mathcal{(N\wedge X)\vee Y=(N\vee Y)\wedge(X\vee Y)}$$
by Proposition~\ref{neutr}, and we are done. Suppose now that the identity~\eqref{xxy=0} fails in $\mathcal N$.

Recall that~\eqref{sequence} is the shortest deduction of the identity $u=v$ from the identities of the varieties $\mathcal N$ and $\mathcal X$. Hence, for every $i=0,1,\dots,k-1$, the identity $u_i=u_{i+1}$ is false in $\mathcal{COM}$. This allows us to suppose that if $u_i$ is a word of length~3 depending on letters $x$ and $y$ then $u_i\in\{x^2y,xy^2\}$. Put $S=\{x^2y,xy^2\}$. The words $u_0$, $u_1$, \dots, $u_k$ are pairwise distinct, whence at most two of them lie in $S$. Recall that $u_0\equiv x^2y\in S$. The identity $u=u_1$ is satisfied by one of the varieties $\mathcal N$ or $\mathcal X$. If it holds in $\mathcal N$ and $u_1\notin S$ then Corollary~\ref{xxy=w} applies with the conclusion that $\mathcal N$ satisfies the identity~\eqref{xxy=0}. But this is not the case. Further, if the identity $u=u_1$ holds in $\mathcal X$ and $u_1\in S$ then the identity $u_1=u_2$ holds in $\mathcal N$ and $u_2\notin S$. Then Corollary~\ref{xxy=w} applies again and we conclude that the variety $\mathcal N$ satisfies the identity~\eqref{xxy=0}. As we have already noted, this is not the case. Thus either the identity $u=u_1$ holds in $\mathcal N$ and $u_1\in S$ or this identity holds in $\mathcal X$ and $u_1\notin S$. Note that $u_2\notin S$ in the first case because  $u_0,u_1\in S$ here. In both the cases, there exists an identity of the form $w_1=w_2$ such that $w_1\in S$, $w_2\notin S$ and the identity holds in $\mathcal X$ (namely, the identity $u_1=u_2$ in the first case, and the identity $u=u_1$ in the second one). Corollary~\ref{xxy=w} shows that $\mathcal N_1$ satisfies the identity~\eqref{xxy=0}. According to Proposition~\ref{neutr}, this implies that the variety $\mathcal N_1$ is neutral in \textbf{Com}. We use this fact below without special references.

By the hypothesis, the identity $x^2y=v$ holds in the variety $\mathcal Y$. Then Corollary~\ref{xxy=w} implies that either the variety $\mathcal N_2$ satisfies the identity~\eqref{xxy=0} or $v\in W$. In the second  case, the identity $x^2y=v$ is equivalent to~\eqref{xxy=xyy} because it fails in the variety $\mathcal{COM}$. Thus either $\mathcal N_2$ satisfies the identity~\eqref{xxy=0} or $\mathcal Y$ satisfies the identity~\eqref{xxy=xyy}. Consider the second case. Corollary~\ref{xxy=xyy implies xxyz=0} implies that the identity~\eqref{xxyz=0} holds in $\mathcal N_2$ in this case. Besides that, substituting~1 to $y$ in~\eqref{xxy=xyy}, we have that every monoid in $\mathcal Y$ is a band (in particular, each group in $\mathcal Y$ is singleton). We see that $\mathcal G_2=\mathcal T$ and $m_2\le1$ in the considerable case. Thus either $\mathcal N_2$ satisfies the identity~\eqref{xxy=0} or $\mathcal{Y=M\vee N}_2$ where $\mathcal M$ is one of the varieties $\mathcal T$ or $\mathcal{SL}$.

Put $\mathcal Z_1=\mathcal{(N\wedge X)\vee Y}$ and $\mathcal Z_2=\mathcal{(N\vee Y)\wedge(X\vee Y)}$. In view of Lemma~\ref{decomposition}, it suffices to verify that $\Gr(\mathcal Z_1)=\Gr(\mathcal Z_2)$, $m(\mathcal Z_1)=m(\mathcal Z_2)$ and $\Nil(\mathcal Z_1)=\Nil(\mathcal Z_2)$. Clearly, the variety $\mathcal C_m\vee\mathcal U$ is combinatorial whenever $m\ge0$ and $\mathcal U$ is an arbitrary nil-variety. Using Lemma~\ref{Gr of join}, we have
\begin{align*}
\Gr(\mathcal Z_1)&=\Gr\bigl(\mathcal G_2\vee\mathcal C_{m_2}\vee\mathcal N_2\vee(\mathcal{N\wedge X)\bigr)=G}_2,\\
\Gr(\mathcal Z_2)&=\Gr\bigl(\mathcal{(N\vee Y)\wedge(X\vee Y)}\bigr)=\mathcal{\Gr(N\vee Y)\wedge\Gr(X\vee Y)}\\
&=\Gr(\mathcal G_2\vee\mathcal C_{m_2}\vee\mathcal N_2\vee\mathcal N)\wedge\Gr(\mathcal G_1\vee\mathcal G_2\vee\mathcal C_{m_1}\vee\mathcal C_{m_2}\vee\mathcal N_1\vee\mathcal N_2)\\
&=\mathcal G_2\wedge(\mathcal G_1\vee\mathcal G_2)=\mathcal G_2\ldotp
\end{align*}
Thus $\Gr(\mathcal Z_1)=\Gr(\mathcal Z_2)$. Further, using Lemma~\ref{m(G+C_m+N)}, we have
\begin{align*}
m(\mathcal Z_1)&=m\bigl(\mathcal{(N\wedge X)\vee Y}\bigr)=m\bigl(\mathcal G_2\vee\mathcal C_{m_2}\vee\mathcal N_2\vee(\mathcal{N\wedge X})\bigr)=m_2,\\
m(\mathcal Z_2)&=m\bigl(\mathcal{(N\vee Y)\wedge(X\vee Y)}\bigr)=\min\bigl\{m(\mathcal{N\vee Y}),m(\mathcal{X\vee Y})\bigr\}\\
&=\min\bigl\{m(\mathcal G_2\vee\mathcal C_{m_2}\vee\mathcal N_2\vee\mathcal N),m(\mathcal G_1\vee\mathcal G_2\vee\mathcal C_{m_1}\vee\mathcal C_{m_2}\vee\mathcal N_1\vee\mathcal N_2)\bigr\}\\
&=\min\bigl\{m(\mathcal G_2\vee\mathcal C_{m_2}\vee\mathcal N_2\vee\mathcal N),m(\mathcal G_1\vee\mathcal G_2\vee\mathcal C_{\max\{m_1,m_2\}}\vee\mathcal N_1\vee\mathcal N_2)\bigr\}\\
&=\min\bigl\{m_2,\max\{m_1,m_2\}\bigr\}=m_2\ldotp
\end{align*}
Thus $m(\mathcal Z_1)=m(\mathcal Z_2)$.

It remains to check that $\Nil(\mathcal Z_1)=\Nil(\mathcal Z_2)$. Put
$$\mathcal I=\var\{x^2y=xy^2,\,x^2yz=0,\,xy=yx\}\ldotp$$
As we have seen above, the varieties $\mathcal N_1$ and $\mathcal N_2$ satisfy the identity~\eqref{xxy=xyy} and therefore, the identity~\eqref{xxyz=0} (see Corollary~\ref{xxy=xyy implies xxyz=0}). In other words, $\mathcal N_1,\mathcal N_2\subseteq\mathcal I$. Simple calculations based on Lemma~\ref{splitting} show that proper subvarieties of the variety $\mathcal I$ are exhausted by the following varieties:
\begin{align*}
\mathcal I_n&=\var\{x^2yz=x_1x_2\cdots x_n=0,\,x^2y=xy^2,\,xy=yx\}\enskip\text{where }n\ge4,\\
\mathcal J&=\var\{x^2yz=x^3=0,\,x^2y=xy^2,\,xy=yx\},\\
\mathcal J_n&=\var\{x^2yz=x^3=x_1x_2\cdots x_n=0,\,x^2y=xy^2,\,xy=yx\}\enskip\text{where }n\ge4,\\
\mathcal K&=\var\{x^2y=0,\,xy=yx\},\\
\mathcal K_n&=\var\{x^2y=x_1x_2\cdots x_n=0,\,xy=yx\}\enskip\text{where }n\ge3,\\
\mathcal L&=\var\{x^2=0,\,xy=yx\},\\
\mathcal L_n&=\var\{x^2=x_1x_2\cdots x_n=0,\,xy=yx\}\enskip\text{where }n\in\mathbb N\ldotp
\end{align*}
This implies that the lattice $L(\mathcal I)$ has the form shown on Fig.~\ref{L(I)}.

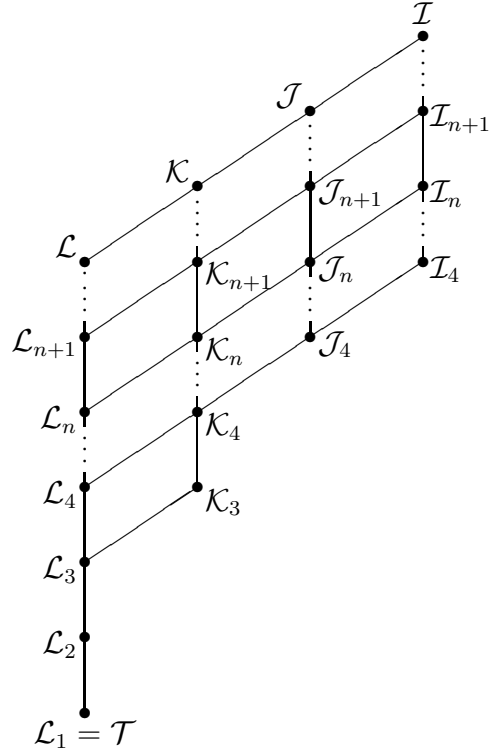
\begin{figure}[tbh]
\begin{center}
\unitlength=1mm
\linethickness{0.4pt}
\begin{picture}(51,98)
\put(5,5){\line(0,1){32}}
\put(5,25){\line(3,2){15}}
\put(5,35){\line(3,2){45}}
\put(5,43){\line(0,1){14}}
\put(5,45){\line(3,2){45}}
\put(5,55){\line(3,2){45}}
\put(5,65){\line(3,2){45}}
\put(20,35){\line(0,1){12}}
\put(20,53){\line(0,1){14}}
\put(35,55){\line(0,1){2}}
\put(35,63){\line(0,1){14}}
\put(50,65){\line(0,1){2}}
\put(50,73){\line(0,1){14}}
\put(5,5){\circle*{1.33}}
\put(5,15){\circle*{1.33}}
\put(5,25){\circle*{1.33}}
\put(5,35){\circle*{1.33}}
\put(5,45){\circle*{1.33}}
\put(5,55){\circle*{1.33}}
\put(5,65){\circle*{1.33}}
\put(20,35){\circle*{1.33}}
\put(20,45){\circle*{1.33}}
\put(20,55){\circle*{1.33}}
\put(20,65){\circle*{1.33}}
\put(20,75){\circle*{1.33}}
\put(35,55){\circle*{1.33}}
\put(35,65){\circle*{1.33}}
\put(35,75){\circle*{1.33}}
\put(35,85){\circle*{1.33}}
\put(50,65){\circle*{1.33}}
\put(50,75){\circle*{1.33}}
\put(50,85){\circle*{1.33}}
\put(50,95){\circle*{1.33}}
\put(5,38.5){\circle*{0.52}}
\put(5,40){\circle*{0.52}}
\put(5,41.5){\circle*{0.52}}
\put(5,58.5){\circle*{0.52}}
\put(5,60){\circle*{0.52}}
\put(5,61.5){\circle*{0.52}}
\put(5,63.5){\circle*{0.52}}
\put(20,48.5){\circle*{0.52}}
\put(20,50){\circle*{0.52}}
\put(20,51.5){\circle*{0.52}}
\put(20,68.5){\circle*{0.52}}
\put(20,70){\circle*{0.52}}
\put(20,71.5){\circle*{0.52}}
\put(20,73){\circle*{0.52}}
\put(35,58.5){\circle*{0.52}}
\put(35,60){\circle*{0.52}}
\put(35,61.5){\circle*{0.52}}
\put(35,78.5){\circle*{0.52}}
\put(35,80){\circle*{0.52}}
\put(35,81.5){\circle*{0.52}}
\put(35,83){\circle*{0.52}}
\put(50,68.5){\circle*{0.52}}
\put(50,70){\circle*{0.52}}
\put(50,71.5){\circle*{0.52}}
\put(50,88.5){\circle*{0.52}}
\put(50,90){\circle*{0.52}}
\put(50,91.5){\circle*{0.52}}
\put(50,93){\circle*{0.52}}
\put(50,98){\makebox(0,0)[cc]{$\mathcal I$}}
\put(51,64){\makebox(0,0)[lc]{$\mathcal I_4$}}
\put(51,74){\makebox(0,0)[lc]{$\mathcal I_n$}}
\put(51,84){\makebox(0,0)[lc]{$\mathcal I_{n+1}$}}
\put(34,87){\makebox(0,0)[rc]{$\mathcal J$}}
\put(36,54){\makebox(0,0)[lc]{$\mathcal J_4$}}
\put(36,64){\makebox(0,0)[lc]{$\mathcal J_n$}}
\put(36,74){\makebox(0,0)[lc]{$\mathcal J_{n+1}$}}
\put(19,77){\makebox(0,0)[rc]{$\mathcal K$}}
\put(21,33){\makebox(0,0)[lc]{$\mathcal K_3$}}
\put(21,43){\makebox(0,0)[lc]{$\mathcal K_4$}}
\put(21,53){\makebox(0,0)[lc]{$\mathcal K_n$}}
\put(21,63){\makebox(0,0)[lc]{$\mathcal K_{n+1}$}}
\put(4,67){\makebox(0,0)[rc]{$\mathcal L$}}
\put(5,2){\makebox(0,0)[cc]{$\mathcal L_1=\mathcal T$}}
\put(4,14){\makebox(0,0)[rc]{$\mathcal L_2$}}
\put(4,24){\makebox(0,0)[rc]{$\mathcal L_3$}}
\put(4,34){\makebox(0,0)[rc]{$\mathcal L_4$}}
\put(4,44){\makebox(0,0)[rc]{$\mathcal L_n$}}
\put(4,54){\makebox(0,0)[rc]{$\mathcal L_{n+1}$}}
\end{picture}
\caption{The lattice $L(\mathcal I)$}
\label{L(I)}
\end{center}
\end{figure}

Identities of the form $w=0$ are called 0-\emph{reduced}. For a commutative nil-variety of semigroups $\mathcal V$, we denote by $\ZR(\mathcal V)$ the variety given by the commutative law and all 0-reduced identities that hold in $\mathcal V$. The exponent of a periodic group variety $\mathcal H$ is denoted by $\exp(\mathcal H)$. To verify the equality $\Nil(\mathcal Z_1)=\Nil(\mathcal Z_2)$, we need two auxiliary facts.

\begin{lemma}
\label{Nil(G+C_m+U)}
Let $\mathcal G$ be a periodic group variety and $\mathcal U$ be a nil-variety of semigroups with $\mathcal{U\subseteq I}$ and $\mathcal U\supseteq\Nil(\mathcal C_m)$ for some $m\le2$. Then
\begin{itemize}
\item[\textup{a)}]$\Nil(\mathcal{G\vee\mathcal C}_m\vee\mathcal{U)\subseteq\ZR(U)}$;
\item[\textup{b)}]if $\mathcal{U\subseteq K}$ then $\Nil(\mathcal{G\vee C}_m\vee\mathcal{U)=U}$.
\end{itemize}
\end{lemma}

\begin{proof}
a) Put $\mathcal{Z=G\vee C}_m\vee\mathcal U$. Let $w=0$ be an arbitrary 0-reduced identity that holds in the variety $\mathcal U$. Because $\mathcal{U\subseteq I}$, we have that $w$ is one of the words $x^2yz$, $x^2y$, $x^3$, $x^2$ or $x_1x_2\cdots x_n$ for some natural $n$ (see Fig.~\ref{L(I)}). Put $r=\exp(\mathcal G)$. If $w\in\{x^2yz,x^2y,x^3,x^2\}$ then the variety $\mathcal Z$ satisfies the identity $x^rw=w$. Then Lemma~\ref{splitting}(ii) applies with the conclusion that the identity $w=0$ holds in the variety $\Nil(\mathcal Z)$. Suppose now that $w\equiv x_1x_2\cdots x_n$. In other words, $\mathcal U$ satisfies the identity~\eqref{x_1...x_n=0}. Because $\mathcal{\Nil(Z)\supseteq\Nil(C}_m)$ and the variety $\Nil(\mathcal C_m)$ with $m>1$ does not satisfy the identity~\eqref{x_1...x_n=0}, we have that $m\le1$ in this case. Then the variety $\mathcal Z$ satisfies the identity $x_1x_2\cdots x_n=x_1^{r+1}x_2\cdots x_n$. Using Lemma~\ref{splitting}(ii) again, we have that the variety $\Nil(\mathcal Z)$ satisfies the identity~\eqref{x_1...x_n=0}. We see that if a 0-reduced identity holds in $\mathcal U$ then it holds in $\Nil(\mathcal Z)$ as well. We prove that $\mathcal{\Nil(Z)\subseteq\ZR(U)}$.

\medskip

b) Let now $\mathcal{U\subseteq K}$. All subvarieties of the variety $\mathcal K$ is given within $\mathcal{COM}$ by 0-reduced identities only (see Fig.~\ref{L(I)}). Therefore $\mathcal{\ZR(U)=U}$. Now the claim~a) implies that $\Nil(\mathcal{G\vee C}_m\vee\mathcal{U)\subseteq U}$. The converse inclusion is evident.
\end{proof}

\begin{lemma}
\label{eliminate ZR}
If $\mathcal U_1,\mathcal U_2\subseteq\mathcal I$ then $\ZR(\mathcal U_1)\wedge\mathcal U_2=\mathcal U_1\wedge\mathcal U_2$.
\end{lemma}

\begin{proof}
Put $\mathcal Q=\var\{x^2y=xy^2,\,xy=yx\}$. Then $\mathcal U_1=\mathcal{Q\wedge\ZR(U}_1)$ (see Fig.~\ref{L(I)}) and $\mathcal U_2\subseteq\mathcal{I\subseteq Q}$. Therefore $\mathcal U_1\wedge\mathcal U_2=\mathcal{Q\wedge\ZR(U}_1)\wedge\mathcal U_2=\ZR(\mathcal U_1)\wedge\mathcal U_2$.
\end{proof}

Now we start with the proof of the equality $\Nil(\mathcal Z_1)=\Nil(\mathcal Z_2)$. Note that if $m>2$ then the variety $\Nil(\mathcal C_m)=\var\{x^m=0,\,xy=yx\}$ does not satisfy the identity~\eqref{xxyz=0}. Since $\Nil(\mathcal C_{m_1})\subseteq\mathcal N_1\subseteq\mathcal X$ and $\Nil(\mathcal C_{m_2})\subseteq\mathcal N_2\subseteq\mathcal Y$, we have $m_1,m_2\le2$. Below we use this fact without special references.

Further, we note that $\mathcal{N\wedge X=N\wedge\Nil(X)=N\wedge N}_1$, whence
\begin{equation}
\label{Z_1}
\mathcal Z_1=\mathcal{(N\wedge N}_1)\vee\mathcal Y\ldotp
\end{equation} 
Suppose at first that the variety $\mathcal N_2$ satisfies the identity~\eqref{xxy=0}. Using the equality~\eqref{Z_1}, we have
$$\mathcal Z_1=\mathcal{(N\wedge N}_1)\vee\mathcal{Y=(N\wedge N}_1)\vee\mathcal N_2\vee\mathcal G_2\vee\mathcal C_{m_2}$$
where $m_2\le2$. Recall that $\mathcal N_1$ satisfies the identity~\eqref{xxy=0}. Now Lemma~\ref{Nil(G+C_m+U)}b) with $\mathcal {U=(N\wedge N}_1)\vee\mathcal N_2$, $\mathcal{G=G}_2$ and $m=m_2$ applies with the conclusion that
\begin{equation}
\label{summary1}
\Nil(\mathcal Z_1)=(\mathcal{N\wedge N}_1)\vee\mathcal N_2\ldotp
\end{equation}
Applying Proposition~\ref{neutr}, we have
\begin{equation}
\label{summary2}
\Nil(\mathcal Z_1)=(\mathcal{N\vee N}_2)\wedge(\mathcal N_1\vee\mathcal N_2)\ldotp
\end{equation}
One can consider the variety $\Nil(\mathcal Z_2)$ now. Since $\mathcal Z_2=\mathcal{(N\vee Y)\wedge(X\vee Y)}$, we have
\begin{equation}
\label{Nil(Z_2)}
\Nil(\mathcal Z_2)=\mathcal{\Nil(N\vee Y)\wedge\Nil(X\vee Y)}\ldotp
\end{equation}
Further, $\mathcal{\Nil(N\vee Y)=\Nil(N\vee N}_2\vee\mathcal G_2\vee\mathcal C_{m_2})$. Now we may apply Lemma~\ref{Nil(G+C_m+U)}a) with $\mathcal{U=N\vee N}_2$, $\mathcal{G=G}_2$ and $m=m_2$, and conclude that
$$\mathcal{\Nil(N\vee Y)\subseteq\ZR(N\vee N}_2)\ldotp$$
On the other hand,
$$\mathcal{X\vee Y=G}_1\vee\mathcal G_2\vee\mathcal C_{m_1}\vee\mathcal C_{m_2}\vee\mathcal N_1\vee\mathcal N_2=\mathcal G_1\vee\mathcal G_2\vee\mathcal C_{\max\{m_1,m_2\}}\vee\mathcal N_1\vee\mathcal N_2\ldotp$$
Now Lemma~\ref{Nil(G+C_m+U)}b) with $\mathcal{U=N}_1\vee\mathcal N_2$, $\mathcal{G=G}_1\vee\mathcal G_2$ and $m=\max\{m_1,m_2\}$ applies with the conclusion that $\mathcal{\Nil(X\vee Y)=N}_1\vee\mathcal N_2$. Thus
$$\Nil(\mathcal Z_2)=\mathcal{\Nil(N\vee Y)\wedge\Nil(X\vee Y)\subseteq\ZR(N\vee N}_2)\wedge(\mathcal N_1\vee\mathcal N_2)\ldotp$$
By the hypothesis, $\mathcal{N\subseteq I}$. Now Lemma~\ref{eliminate ZR} with $\mathcal U_1=\mathcal{N\vee N}_2$ and $\mathcal U_2=\mathcal N_1\vee\mathcal N_2$ may be applied with the conclusion that
$$\Nil(\mathcal Z_2)\subseteq(\mathcal{N\vee N}_2)\wedge(\mathcal N_1\vee\mathcal N_2)\ldotp$$
Because the converse inclusion is evident, we have
\begin{equation}
\label{summary3}
\Nil(\mathcal Z_2)=(\mathcal{N\vee N}_2)\wedge(\mathcal N_1\vee\mathcal N_2)\ldotp
\end{equation}
The equalities~\eqref{summary2} and~\eqref{summary3} imply that $\Nil(\mathcal Z_1)=\Nil(\mathcal Z_2)$.

It remains to consider the case when $\mathcal N_2$ does not satisfy the identity~\eqref{xxy=0}. Recall that $\mathcal{Y=M\vee N}_2$ where $\mathcal M$ is one of the varieties $\mathcal T$ or $\mathcal{SL}$ in this case. The equality~\eqref{Z_1} implies that
$$\mathcal Z_1=\mathcal{(N\wedge N}_1)\vee\mathcal{Y=(N\wedge N}_1)\vee\mathcal N_2\vee\mathcal M$$
where $\mathcal M$ has the just mentioned sense. Lemma~\ref{dir prod} implies that the equality~\eqref{summary1} holds. This equality and Proposition~\ref{neutr} show that the equality~\eqref{summary2} is true. Besides that, the equality~\eqref{Nil(Z_2)} holds. Suppose that $\mathcal{M=SL}$. Proposition~\ref{neutr} shows that
\begin{align*}
\mathcal{(N\vee Y)\wedge(X\vee Y)}&=(\mathcal{N\vee N}_2\vee\mathcal{SL)\wedge(X\vee N}_2\vee\mathcal{SL})\\
&=\bigl((\mathcal{N\vee N}_2)\wedge(\mathcal{X\vee N}_2)\bigr)\vee\mathcal{SL}\ldotp
\end{align*}
Now we may apply Lemma~\ref{dir prod} and conclude that
$$\Nil(\mathcal Z_2)=\Nil\bigl((\mathcal{N\vee N}_2)\wedge(\mathcal{X\vee N}_2)\bigr)\ldotp$$
Clearly, this equality holds whenever $\mathcal{M=T}$ too. Thus it is valid always. Note that
$$\mathcal{X\vee N}_2=\mathcal G_1\vee\mathcal C_{m_1}\vee\mathcal N_1\vee\mathcal N_2\ldotp$$
Using Lemma~\ref{Nil(G+C_m+U)}a) with $\mathcal{U=N}_1\vee\mathcal N_2$, $\mathcal{G=G}_1$ and $m=m_1$, we have
$$\Nil(\mathcal{X\vee N}_2)\subseteq\ZR(\mathcal N_1\vee\mathcal N_2)\ldotp$$
Since $\mathcal{N\vee N}_2$ is a nil-variety, we have
\begin{align*}
\Nil(\mathcal Z_2)&=\Nil\bigl((\mathcal{N\vee N}_2)\wedge(\mathcal{X\vee N}_2)\bigr)\\
&=(\mathcal{N\vee N}_2)\wedge\Nil(\mathcal{X\vee N}_2)\\
&\subseteq(\mathcal{N\vee N}_2)\wedge\ZR(\mathcal N_1\vee\mathcal N_2)\ldotp
\end{align*}
Now we may apply Lemma~\ref{eliminate ZR} with $\mathcal U_1=\mathcal N_1\vee\mathcal N_2$ and $\mathcal U_2=\mathcal{N\vee\mathcal N}_2$ and conclude that the equality~\eqref{summary3} holds. Because we prove above that the equality~\eqref{summary2} is true, we have $\Nil(\mathcal Z_1)=\Nil(\mathcal Z_2)$.

We complete the proof of Theorem~\ref{costand}.\qed

\medskip

\emph{The implication} c)\,$\longrightarrow$\,b) \emph{of Theorem}~\ref{umod&codistr}. It is evident that the variety $\mathcal{COM}$ is codistributive in \textbf{Com}. If the variety $\mathcal V$ satisfies the claim~(ii) of Theorem~\ref{umod&codistr} then Theorem~\ref{costand} applies with the conclusion that $\mathcal V$ is costandard and therefore, is codistributive in \textbf{Com}. It remains to consider the case when $\mathcal V$ satisfies the claim~(iii) of Theorem~\ref{umod&codistr}. So, let $\mathcal{V=G\vee M\vee N}$ where $\mathcal G$ is an Abelian periodic group variety, $\mathcal M$ is one of the varieties $\mathcal T$, $\mathcal{SL}$ or $\mathcal C_2$, and $\mathcal N$ is a commutative variety satisfying the identity~\eqref{xxy=0}. 

Let $\mathcal X$ and $\mathcal Y$ be arbitrary commutative semigroup varieties. It remains to verify that $\mathcal{V\wedge(X\vee Y)\subseteq(V\wedge X)\vee(V\wedge Y)}$ because the converse inclusion is evident. If at least one of the varieties $\mathcal X$ or $\mathcal Y$ coincides with the variety $\mathcal{COM}$ then the desirable inclusion is evident. Thus we may assume that the varieties $\mathcal X$ and $\mathcal Y$ are periodic. Put $\mathcal Z_1=\mathcal{V\wedge(X\vee Y)}$ and $\mathcal Z_2=\mathcal{(V\wedge X)\vee(V\wedge Y)}$. The varieties $\mathcal Z_1$ and $\mathcal Z_2$ are periodic. In view of Lemma~\ref{decomposition}, $\mathcal Z_1=\mathcal G_1\vee\mathcal C_{m_1}\vee\mathcal N_1$ and $\mathcal Z_2=\mathcal G_2\vee\mathcal C_{m_2}\vee\mathcal N_2$ for some Abelian periodic group varieties $\mathcal G_1$ and $\mathcal G_2$, some $m_1,m_2\ge0$ and some nil-varieties $\mathcal N_1$ and $\mathcal N_2$. We may assume without loss of generality that $\mathcal G_i=\Gr(\mathcal Z_i)$ and $\mathcal N_i=\Nil(\mathcal Z_i)$ for $i=1,2$. If $m>2$ then the variety $\Nil(\mathcal C_m)$ does not satisfy the identity~\eqref{xxy=0}. Therefore $m_1,m_2\le2$.

Clearly, it suffices to verify that $\mathcal G_1=\mathcal G_2$, $m(\mathcal Z_1)=m(\mathcal Z_2)$ and $\mathcal N_1\subseteq\mathcal N_2$. Put $q=\exp\bigl(\Gr(\mathcal V)\bigr)$, $r=\exp\bigl(\Gr(\mathcal X)\bigr)$ and  $s=\exp\bigl(\Gr(\mathcal Y)\bigr)$. Then
$$\exp(\mathcal G_1)=\gcd\bigl(q,\lcm(r,s)\bigr)\enskip\text{and}\enskip\exp(\mathcal G_2)=\lcm\bigl(\gcd(q,r),\gcd(q,s)\bigr)\ldotp$$
Since the lattice of all natural numbers with the operations $\gcd$ and $\lcm$ is distributive, we have that $\exp(\mathcal G_1)=\exp(\mathcal G_2)$. This implies that $\mathcal G_1=\mathcal G_2$ because the varieties $\mathcal G_1$ and $\mathcal G_2$ are Abelian.

Put $m(\mathcal V)=k$, $m(\mathcal X)=\ell$ and $m(\mathcal Y)=m$. It is clear that
$$m(\mathcal{E\wedge F})=\min\bigl\{m(\mathcal E),m(\mathcal F)\bigr\}$$
for arbitrary periodic varieties $\mathcal E$ and $\mathcal F$. Combining this observation with Corollary~\ref{m(X+Y)}, we have that
$$m(\mathcal Z_1)=\min\bigl\{k,\max\{\ell,m\}\bigr\}\enskip\text{and}\enskip m(\mathcal Z_2)=\max\bigl\{\min\{k,\ell\},\min\{k,m\}\bigr\}\ldotp$$
This implies that $m(\mathcal Z_1)=m(\mathcal Z_2)$.

It remains to verify that $\mathcal N_1\subseteq\mathcal N_2$. It is evident that $\mathcal N_1,\mathcal N_2\subseteq\Nil(\mathcal V)$. The variety $\mathcal V$ is commutative and satisfies the identity $x^2y=x^{r+2}y$ where $r=\exp(\mathcal G)$. Lemma~\ref{splitting}(ii) implies now that $\mathcal N_1$ and $\mathcal N_2$ satisfy the identity~\eqref{xxy=0}. This means that $\mathcal N_1,\mathcal N_2\subseteq\mathcal K$. Every subvariety of the variety $\mathcal K$ is given within $\mathcal K$ either by the identity
\begin{equation}
\label{xx=0}
x^2=0
\end{equation}
or by the identity~\eqref{x_1...x_n=0} for some $n$ or by these two identities simultaneously (see Fig.~\ref{L(I)}). Thus it suffices to verify that $\deg(\mathcal Z_1)=\deg(\mathcal Z_2)$ and the identity~\eqref{xx=0} holds in the variety $\mathcal N_1$ whenever it holds in $\mathcal N_2$.

Put $\deg(\mathcal V)=k$, $\deg(\mathcal X)=\ell$ and $\deg(\mathcal Y)=m$. It is evident that
$$\mathcal{\deg(E\wedge F)=\min\bigl\{\deg(E),\deg(F)\bigr\}}$$
for arbitrary semigroup varieties $\mathcal E$ and $\mathcal F$. Combining this observation with Corollary~\ref{deg of join}, we have that
$$\deg(\mathcal Z_1)=\min\bigl\{k,\max\{\ell,m\}\bigr\}\enskip\text{and}\enskip\deg(\mathcal Z_2)=\max\bigl\{\min\{k,\ell\},\min\{k,m\}\bigr\}\ldotp$$
This implies that $\deg(\mathcal Z_1)=\deg(\mathcal Z_2)$.

Suppose now that $\mathcal N_2$ satisfies the identity~\eqref{xx=0}. Being periodic, the variety $\mathcal Z_2$ satisfies the identity $x^n=x^m$ for some natural numbers $n$ and $m$ with $m>n$. Let $n$ be the least number with such property. Then Lemma~\ref{splitting}(ii) implies that the variety $\mathcal N_2=\Nil(\mathcal Z_2)$ satisfies the identity $x^n=0$ and $n$ is the least number with such a property. Hence $n\le2$. Thus the variety $\mathcal Z_2=\mathcal{(V\wedge X)\vee(V\wedge Y)}$ satisfies the identity $x^2=x^m$ for some $m>2$. In particular, this identity holds in the variety $\mathcal{V\vee X}$. Therefore there exists a deduction of this identity from the identities of the varieties $\mathcal V$ and $\mathcal X$. In particular, one of these varieties satisfies a non-trivial identity of the form $x^2=w$. Now Lemma~\ref{splitting} implies that one of the varieties $\Nil(\mathcal V)$ or $\Nil(\mathcal X)$ satisfies the identity~\eqref{xx=0}. If this identity holds in $\Nil(\mathcal V)$ then it holds in the variety $\mathcal{\Nil\bigl(V\wedge(X\vee Y)\bigr)=N}_1$ too. Thus we may assume that the identity~\eqref{xx=0} is satisfied by the variety $\Nil(\mathcal X)$. Analogously, using a deduction of the identity $x^2=x^m$ from the identities of the varieties $\mathcal V$ and $\mathcal Y$, we may reduce our considerations to the case when the identity~\eqref{xx=0} holds in $\Nil(\mathcal Y)$. The same arguments as we use at the beginning of this paragraph allows us to check that the varieties $\mathcal X$ and $\mathcal Y$ satisfy, respectively, the identities $x^2=x^{q+2}$ and $x^2=x^{r+2}$ for some natural numbers $q$ and $r$. Therefore $\mathcal{X\vee Y}$ satisfies the identity $x^2=x^{qr+2}$. Then Lemma~\ref{splitting}(ii) implies that the variety $\Nil(\mathcal{X\vee Y})$ satisfies the identity~\eqref{xx=0}. Then it holds in the variety $\mathcal{\Nil\bigl(V\wedge(X\vee Y)\bigr)=N}_1$ too.

We complete the proof of Theorem~\ref{umod&codistr}.\qed

\section{Corollaries}
\label{corollary}

One can give several corollaries of main results. Theorem~\ref{umod&codistr} and~\cite[Theorem~1.2]{Vernikov-08a} imply

\begin{corollary}
\label{Com and SEM}
A commutative semigroup variety $\mathcal V$ with $\mathcal{V\ne COM}$ is an upper-modular element of the lattice $\mathbf{Com}$ if and only if it is an upper-modular element of the lattice $\mathbf{SEM}$.\qed
\end{corollary}

Comparing Theorems~\ref{umod&codistr} and~\ref{costand}, we have

\begin{corollary}
\label{umod&codistr&costand}
For a commutative nil-variety of semigroups $\mathcal V$, the following are equivalent:
\begin{itemize}
\item[\textup{a)}]$\mathcal V$ is an upper-modular element of the lattice $\mathbf{Com}$;
\item[\textup{b)}]$\mathcal V$ is a codistributive element of the lattice $\mathbf{Com}$;
\item[\textup{c)}]$\mathcal V$ is a costandard element of the lattice $\mathbf{Com}$;
\item[\textup{d)}]$\mathcal V$ satisfies the identities~\eqref{xxyz=0} and~\eqref{xxy=xyy}.\qed
\end{itemize}
\end{corollary}

Theorem~\ref{umod&codistr} implies

\begin{corollary}
\label{umod is hered}
If a commutative semigroup variety $\mathcal V$ is an upper-modular element of the lattice $\mathbf{Com}$ and $\mathcal{V\ne COM}$ then every subvariety of the variety $\mathcal V$ is an upper-modular element of the lattice $\mathbf{Com}$.\qed
\end{corollary}

Note that the analog of this assertion for the lattice \textbf{SEM} is the case (see~\cite[Corollary~3]{Vernikov-08b}). Theorem~\ref{umod&codistr} and results of the article~\cite{Volkov-91} imply

\begin{corollary}
\label{umod implies distr lat}
If a commutative semigroup variety $\mathcal V$ is an upper-modular element of the lattice $\mathbf{Com}$ and $\mathcal{V\ne COM}$ then the lattice $L(\mathcal V)$ is distributive.\qed
\end{corollary}

We does not know, whether the analog of this fact in the lattice \textbf{SEM} is true. It is verified in~\cite[Corollary~2]{Vernikov-08b} that the following weaker statement is the case: if a variety $\mathcal V$ is an upper-modular element of the lattice \textbf{SEM} and $\mathcal V$ is not the variety of all semigroups then the lattice $L(\mathcal V)$ is modular.

\end{document}